\documentclass[10pt,reqno,oneside]{amsproc}

\usepackage{amsfonts}
\usepackage{dsfont}
\usepackage[usenames,dvipsnames,svgnames,table]{xcolor}
\usepackage{amsmath, amsthm, amssymb, enumerate}

\usepackage{color}  
\usepackage{marginnote}
\usepackage[colorlinks=true, pdfstartview=FitV, linkcolor=black, citecolor=black, urlcolor=black]{hyperref}

  \chardef\forshowkeys=0
  \chardef\refcheck=0
  \chardef\showllabel=0
  \chardef\sketches=0
  \chardef\showcolors=0

\ifnum\forshowkeys=1
  
  \usepackage[notref,notcite,color]{showkeys}
\fi

\ifnum\showllabel=1
\def\llabel#1{\marginnote{\color{gray}\rm(#1)}[-0.0cm]\notag}
\else
\def\llabel#1{\notag}
\fi

 \ifnum\refcheck=1
  \usepackage{refcheck}
\fi

\newtheorem{theorem}{Theorem}[section]

\newtheorem{Lemma}[theorem]{Lemma}

\theoremstyle{definition}

\newtheorem{Remark}[theorem]{Remark}

\def\lec{\lesssim}

   \def\Pas{\indeq\mathbb{P}\text{-a.s.}}

   \def\PP{{\mathbb P}}

   \def\um{u^{(m)}}
   \def\umm{u^{(m-1)}}

   \def\uu{{{u}}}

   \def\startnewsection#1#2{\section{#1}\label{#2}\setcounter{equation}{0}}   
   \def\NNp{{\mathbb N}}

    \def\TT{{\mathbb T}}
   \def\WW{{\mathbb W}}
   \def\EE{{\mathbb E}}
   \def\comma{ {\rm ,\qquad{}} }            
   \def\commaone{ {\rm ,\quad{}} }         
   
   \def\indeq{\qquad{}}                     

\def\TT{\mathbb T}

\def\tilde{\widetilde}

\def\PP{\mathbb{P}}

\def\indeq{\quad{}}

\ifnum\showcolors=1
  \def\cole{\color{coloroftheorems}}
  \definecolor{colorcccc}{rgb}{0.7,0.7,0.7}

  \def\colb{\color{black}}
  \definecolor{colorpppp}{rgb}{0.6,0.0,0.1}
  \definecolor{colorgggg}{rgb}{.0,0.4,0.0}
  \definecolor{colorhhhh}{rgb}{0,0.6,0.2}
  \definecolor{colorgray}{rgb}{0.8,0.8,0.8}

  \definecolor{coloroftheorems}{rgb}{0.6,0.0,0.6}
  \definecolor{colorigor}{rgb}{1, 0.2, 0.8}
  \definecolor{amethyst}{rgb}{0.6, 0.4, 0.8}

  \definecolor{colororange}{rgb}{0.8,0.2,0}
  \definecolor{colorpurple}{rgb}{0.6,0.0,0.6}
  
\else   
  \def\cole{}
  \definecolor{colorcccc}{rgb}{0,0,0}

  \def\colb{\color{black}}
  \definecolor{colorpppp}{rgb}{0,0,0}
  \definecolor{colorgggg}{rgb}{0,0,0}
  \definecolor{colorhhhh}{rgb}{0,0,0}
  \definecolor{colorgray}{rgb}{0,0,0}

  \definecolor{coloroftheorems}{rgb}{0,0,0}
  \definecolor{colorigor}{rgb}{0,0,0}
  \definecolor{amethyst}{rgb}{0,0,0}
  
  \def\cole{\color{coloroftheorems}}
  
  \definecolor{colororange}{rgb}{0.8,0.2,0}
  \definecolor{colorpurple}{rgb}{0.6,0.0,0.6}
  
\fi

  \def\bea{\begin{align}}
  \def\ena{\end{align}}

\def\bega{\begin{aligned}}
  \def\enda{\end{aligned}}

\def\bcase{\begin{cases}}
  \def\ecase{\end{cases}}

\def\bmx{\begin{bmatrix}}
  \def\emx{\end{bmatrix}}

\def\um{u^{(m)}}

\def\uu{{{u}}}
\def\WW{ W}

\def\NNp{{\mathbb N}}

\begin{document}
\baselineskip=12.6pt

$\,$
\vskip1.2truecm
\title[The Navier-Stokes equations with transport noise in critical $H^{1/2}$ space]{The Navier-Stokes equations with transport noise in critical $H^{1/2}$ space}

\author[M.S.~Ayd\i n]{Mustafa Sencer Ayd\i n}
\address{Department of Mathematics, University of Southern California, Los Angeles, CA 90089}
\email{maydin@usc.edu}


\author[F.H.~Xu]{Fanhui Xu}
\address{Department of Mathematics, Amherst College, Amherst, MA 01002}
\email{fxu@amherst.edu}


\begin{abstract}
	We study the Navier-Stokes equations with transport noise in critical function spaces. Assuming the initial data belongs to $H^{1/2}$
	almost surely, we establish the existence and uniqueness of a local-in-time probabilistically strong solution. Moreover, we show that the probability of global existence can be made arbitrarily close to $1$ by choosing the initial data norm sufficiently small, and that the solution norm remains small for all time. Our analysis is independent of the compactness of the spatial domain, and consequently, the results apply both to the three-dimensional torus and to the whole space. 
	\hfill 
	\today
\end{abstract}

\maketitle

\date{}

\startnewsection{Introduction}{sec01}

In this work, we aim to solve the initial value problem for the following stochastic Navier-Stokes equations (SNSE),
\begin{align} 
	\begin{split}
		&	(\partial_t - \Delta) u = -\mathcal{P}\nabla \cdot  (u \otimes u) +  \mathcal{P}((b\cdot \nabla) u)\dot{W}
		\\
		&\nabla\cdot u = 0,\\
		&u(0)=u_0\comma t\geq 0, ~x\in \TT^3
		,
	\end{split}
	\label{EQ01}
\end{align}
where $u$ denotes the velocity field of a viscous, incompressible Newtonian flow, and $\mathcal{P}$ is the Leray projector onto the divergence and average-free fields, which eliminates the pressure gradient from the original Navier-Stokes equations. In the stochastic term, we take $W$ to be a standard real-valued Brownian motion defined on a given stochastic basis $(\Omega, \mathcal{F}, (\mathcal{F}_t)_{t\geq 0}, \mathbb{P})$. We let the transport field $b$ satisfy $\nabla \cdot b =0$ and 
\begin{align}
	\Vert (b\cdot \nabla) f\Vert_{H^{s}}
	 \le \epsilon_b \Vert f\Vert_{H^{s+1}}\comma s\in \left\{0, 1/2\right\},
	 \label{EQ14}
\end{align}
with $\epsilon_b >0$ to be chosen small, and we assume that the initial datum $u_0$ is divergence- and average-free, and small in $H^{1/2}(\TT^3)$ almost surely. 

The original motivation for formulating the SNSE was to introduce random forcing to model environmental perturbations and turbulent fluctuations in fluid flows (see~\cite{BeT, BCF, CC}). In particular, the SNSE with transport-type noise is obtained by perturbing the advecting velocity field, and it serves as a representation of the impact of small-scale turbulence on large-scale fluid dynamics (see~\cite{CCD, MR, MR2, FP} and references therein). A seminal result by Flandoli, Gubinelli, and Priola~\cite{FGP} showed that transport-type noise can restore well-posedness for linear transport equations that are ill-posed in the deterministic setting. This naturally leads to the broader question of whether such noise can promote well-posedness for nonlinear PDEs, particularly those arising in fluid dynamics.

The answer can be either positive or negative. For example, Coghi and Maurelli~\cite{MM} showed that transport noise restores uniqueness for the 2D Euler equations with $L^p$ initial data for suitable~$p$. In contrast, Hofmanov\'{a}, Lange, and Pappalettera~\cite{HLP} proved that, for infinitely many H\"older-continuous initial data, the 3D Euler equations with transport noise lack uniqueness in law. In the viscous case, Flandoli and Luo~\cite{FL} studied the vorticity formulation of the Navier–Stokes equations with transport noise and showed that such noise imposes a bound on vorticity, yielding well-posedness with high probability. Subsequently, Luo~\cite{L} demonstrated dissipation enhancement by transport noise for the 2D SNSE in velocity form and suppression of blow-up by transport noise for the 3D SNSE in vorticity form. More recently, Agresti~\cite{Aa} established the global existence of smooth solutions, with high probability, for the Navier–Stokes equations with hyperviscosity. Similar phenomena also occur in other nonlinear PDEs: for instance, Flandoli, Gubinelli, and Luo~\cite{FGL} proved that transport-type noise can suppress blow-up in the 3D Keller-Segel and Fisher-KPP equations, thereby ensuring long-time existence of solutions with high probability.

In this paper, we consider the 3D SNSE with transport noise in the velocity form, focusing on the critical $H^{1/2}$ space. Recall that the Navier-Stokes equations possess the following scaling: if $u(x,t)$ solves the equations with initial datum $u_0(x)$, then 
$\lambda u(\lambda x, \lambda^2 t)$ is a solution corresponding to initial datum $\lambda u_0 (\lambda x)$.
A function space is called \textit{critical} if it is invariant under this scaling. Examples of critical spaces for the three-dimensional Navier-Stokes equations include
\begin{align*}
	\dot{H}^\frac12 \subset L^3 
	 \subset \dot{B}^{-1+\frac{3}{p}}_{p,\infty}
	  \subset \text{BMO}^{-1}
	   \subset \dot{B}^{-1,\infty}_\infty,
\end{align*} 
where $p<\infty$. These spaces are often viewed as the threshold of regularity at which one may expect well-posedness. 
The first critical-space result for the deterministic Navier–Stokes equations was obtained by Fujita and Kato~\cite{FK}, who established well-posedness in 3D for initial data in $\dot{H}^{1/2}$. They proved local existence of solutions for arbitrary divergence-free data and global existence for sufficiently small data. This result was later extended to $L^3$ by Kato~\cite{K}, who established local existence for arbitrary $L^p$ initial data with $p \geq 3$ and global existence for small $L^p$ data. In the setting of critical Besov spaces, positive results were obtained in~\cite{Pl, C}, while Bourgain and Pavlovi\'c~\cite{BP} showed ill-posedness in the sense of norm inflation in $\dot{B}^{-1,\infty}_\infty$; see also~\cite{L-R} for an extensive exposition. Around the same time, Koch and Tataru~\cite{KT} proved well-posedness for small initial data in $\text{BMO}^{-1}$. Most recently, non-uniqueness of smooth solutions originating from large data in $\text{BMO}^{-1}$ was established in~\cite{CP}.

There has been growing interest in the well-posedness of the SNSE in critical spaces. For example, Kim~\cite{Ki} and Mohan and Sritharan~\cite{MoS} worked toward extending Kato’s framework~\cite{FK, K} to the stochastic setting. In~\cite{Ki}, Kim established the existence of a unique local strong solution to the 3D SNSE in $H^s$ for $s > 1/2$, and proved global existence with high probability for small data. In~\cite{MoS}, the SNSE with L\'evy noise was studied in $L^p(\mathbb{R}^d)$, for $d \geq 2$, $p \geq d$, with initial data in $L^d(\mathbb{R}^d)$, yielding a unique local mild solution. However, in both works, the global existence of a (probabilistically) strong solution at critical values remained out of reach, reflecting the challenges in extending deterministic results to the stochastic setting: on one hand, compactness arguments become more delicate in the presence of stochastic forcing; on the other hand, critical spaces leave little room for embedding inequalities to control the nonlinear term, making it difficult to establish convergence in a strong topology.

In a series of works, the authors of the present paper, together with Kukavica, Wang, and Ziane, have also investigated the SNSE with \textit{state-dependent} noise in nearly critical and critical spaces. Specifically, \cite{KXZ, KX1, KWX} established the existence of a unique local strong solution in $L^p$ for $p > 3$, on the torus and in the whole space respectively, thereby almost recovering the deterministic counterpart. These results were subsequently extended to $L^3$ in \cite{KX2} and to $H^{1/2}$ in \cite{AKX}, where global existence was established on a large part of the probability space for small initial data, by decomposing the equation into a sequence of Navier-Stokes-type equations. In \cite{AKX2}, we removed the smallness assumption and proved local existence in $L^3$ for general $L^3$ initial data $u_0$ by decomposing $u_0=v_0+w_0$ with $\|v_0\|_{L^3}$ sufficiently small and $\|w_0\|_{L^6}$ arbitrary. We also note the contributions of Agresti and Veraar~\cite{AV}, who studied the SNSE with transport noise. They established local well-posedness in a range of Besov spaces and proved almost global results for small initial data. In particular, they considered the space $B^{-1+3/q}_{3,3}$ for a range of $q$ values, including $q = 3$, using maximal regularity techniques.

The present paper establishes \textit{almost global existence} (i.e., global existence with high probability) for the initial value problem~\eqref{EQ01} with small initial data in $H^{1/2}$, and derives \textit{energy-type inequalities} (see~\eqref{EQ65}--\eqref{EQ66} below). We show that \textit{if the solution starts small in $H^{1/2}$, then it remains small in $H^{1/2}$ almost surely, which in turn guarantees pathwise uniqueness}. We note that the model~\eqref{EQ01} is not covered by~\cite{AKX}. Unlike~\cite{AKX}, where approximate models were obtained by decomposing the equation into an infinite sum of difference equations and the initial data into an infinite sum of $H^{\frac{1}{2}+\delta}$ components, we instead work with truncated models and estimate the probability that they coincide with the original model. We also emphasize the methodological differences among works addressing similar problems. In~\cite{AKX}, the difficulty arising from the critical space was overcome by constructing global solutions to the difference equations in $H^{\frac{1}{2}+\delta}$ and subsequently deriving an $H^{1/2}$ bound via the energy dissipation of the deterministic heat equation, while in this work, global solutions to the approximate models are obtained directly in $H^{1/2}$ with the help of cutoff functions. In~\cite{Ki}, both cutoff functions and contraction mappings were utilized to establish global solutions to the truncated model; however, the condition $s>1/2$ played a key role in extending local approximate solutions to global ones. This difficulty does not arise in the present work, since we show that the solution norms remain small for all time, which helps to close the energy estimates. The local-in-time existence result for large data will be addressed in a future work.

In Section~\ref{sec02} below, we introduce the notation for what follows and outline the main results. In Section~\ref{sec03}, we prove the main theorem.

\startnewsection{Preliminaries and Main Results}{sec02}

In this section, we unify the notation for the sequel discussions and state our main results. We use $C$ to denote a generic constant whose value may change from line to line. 
Meanwhile, we write $a \lec b$ to mean that there exists a constant $C>0$ such that $a \leq C b$. 
When it is necessary to emphasize the dependence of $C$ on certain parameters, we indicate this by adding the parameters as subscripts of $C$ or $\lec$.

Recall that, when $u$ is sufficiently smooth, the classical Leray projector $\mathcal{P}$ on $\mathbb{T}^3$ is given by
\begin{align}
	( \mathcal{P} \uu)_j( x)=\sum_{k=1}^{3}( \delta_{jk}+R_j R_k) \uu_k( x)\comma j=1,2,3,
	\llabel{EQ000002}
\end{align} 
with $R_j$'s representing the Riesz transforms. Based on this, we define the average-free Leray projection as 
\begin{equation*}
\mathcal{P}_0 u = \mathcal{P} u - \frac{1}{|\TT^3|}\int_{\TT^3}  \mathcal{P} u \, dx.
\end{equation*}
Note that the Leray projector \( \mathcal{P}\) does not change the spatial average of a function on the torus. Hence, if \(u\) is average-free, then so are \(\nabla \cdot (u \otimes u)\) and \(\mathcal{P}\nabla \cdot (u \otimes u)\). In this case, \(\mathcal{P}\) coincides with \(\mathcal{P}_0\). Moreover,  
\begin{equation*}
\int_{\TT^3} (b \cdot \nabla) u \, dx
= \int_{\TT^3} \nabla \cdot (b\otimes u) \, dx=0
\end{equation*}
if $\nabla \cdot b = 0$. Therefore, 
\begin{equation*}
\int_{\TT^3} \mathcal{P} \big( (b \cdot \nabla) u \big) \, dx=\int_{\TT^3} (b \cdot \nabla) u \, dx = 0.
\end{equation*}

Let $k \in \mathbb{Z}^3$. For $s \in \mathbb{R}$ and $p \in [1,\infty)$, we introduce the Bessel–Sobolev spaces,
\begin{align}
	W^{s,p} = \left\{ u \in \mathcal{S}'(\TT^3) : 
	\bigl((1+|k|^2)^{s/2}\hat{u}\bigr)^{\vee} \in L^p(\TT^3)
	\right\}.
	\label{EQ63}
\end{align}
It is known that for $s\geq 0$ and $1 < p < \infty$, the $W^{s,p}$-norm is equivalent to the standard (Slobodeckij) Sobolev norm. If $s<0$, then \eqref{EQ63} agrees with the dual-space definition. We abbreviate $W^{s,p}$ as $H^s$ for $p=2$.

We will repeatedly apply the Sobolev product and embedding inequalities in the sequel to estimate the advection term, written as:
\begin{align}
	\Vert u \otimes v\Vert_{H^{\frac{1}{2}}}
	\lec 
	\Vert u \Vert_{L^{3}}\Vert v \Vert_{W^{\frac{1}{2}, 6}}
	+\Vert v \Vert_{L^{3}}\Vert u \Vert_{W^{\frac{1}{2}, 6}}
	\lec
	\Vert u\Vert_{H^{\frac12}}\Vert v\Vert_{H^{\frac{3}{2}}}
	+\Vert v\Vert_{H^{\frac12}}\Vert u\Vert_{H^{\frac{3}{2}}}
	.\label{EQ15}
\end{align}
\colb

To establish a probabilistically strong solution, we fix a stochastic basis $(\Omega, \mathcal{F},(\mathcal{F}_t)_{t\geq 0},\mathbb{P})$
that satisfies the standard assumptions. We say that a pair $(u, \tau)$ solves~\eqref{EQ01} in the probabilistically strong sense if $\tau$ is a positive stopping time $\PP$-almost surely and u satisfies the following analytical weak formulation:
\begin{align}
	\begin{split}
	\bigl(u(t)-u_0, \varphi\bigr)=\int_0^t \bigl(u(s), \Delta\varphi\bigr)+\bigl( \mathcal{P}  (u \otimes u),  \nabla \varphi\bigr)\,ds-\int_0^t \bigl(  \mathcal{P} (b \otimes u),     \nabla  \varphi\bigr)\,dW_s, \Pas
	\end{split}
\end{align}
for all $\varphi\in C^{\infty}(\TT^3)$ and all $t\in (0, \tau)$.

Our main theorem is stated below. 
\cole
\begin{theorem}\label{T01}
	Suppose $u_0$ is divergence-free and average-free, and that the parameter $\epsilon_b$ in~\eqref{EQ14} is sufficiently small. Then for every $p_0 \in (0,1)$, there exists $\epsilon_0 \in (0,1)$ such that if $\sup_{\Omega}\|u_0\|_{H^{1/2}} \le \epsilon_0$, the initial value problem~\eqref{EQ01} admits a unique (probabilistically) strong solution $(u,\tau)$ satisfying
	\begin{align}
		\begin{split}
			\EE\biggl[&
			\sup_{0\leq t\leq \tau}
			\Vert u(t)\Vert_{H^{\frac12 }}^2
			+ \int_0^{\tau} \Vert u(t)\Vert_{H^{\frac{3}{2}}}^2 \,dt\biggr]
			\leq C
			\EE \left[	\Vert u_0\Vert_{H^{\frac12 }}^2\right]
		\end{split}
		\label{EQ65}
	\end{align}
	and
	\begin{align}
			\sup_{0\leq t\leq \tau}
			\Vert u(t)\Vert_{H^{\frac12 }}^2
			+ \int_0^{\tau} \Vert u(t)\Vert_{H^{\frac{3}{2}}}^2 \,dt
			\leq C \bar{\epsilon}^2
			\comma \PP ~a.s.,
		\label{EQ66}
	\end{align}
	for some $\bar{\epsilon} \in (\epsilon_0,1)$ and positive constant $C$, with $\PP(\tau < \infty) \le p_0$.
\end{theorem}
\colb

\begin{Remark}
Our proof does not rely on the compactness of $\mathbb{T}^3$, and thus our result also holds for the spatial domain $\mathbb{R}^3$. In particular, we can prove~\cite[Lemma~3.1]{AKX} and Lemma~\ref{L08} for $\mathbb{R}^3$ by replacing the discrete Fourier multiplier with the continuous one.
\end{Remark}

The proof of this theorem is presented in the next section.

\startnewsection{Proof of the Main Results}{sec03}

To address \eqref{EQ01}, we truncate the advection term using a cutoff function $\psi$ defined as follows,
\begin{align}
	\psi(v)
	:=
	\theta\left(
	\frac{1}{\bar{\epsilon}}
	\Vert v(t)\Vert_{H^{\frac{1}{2}}}
	+ \frac{1}{\bar{\epsilon}}\left(\int_0^t \Vert v(s)\Vert_{H^{{\frac{3}{2}}}}^2\,ds\right)^\frac12
	\right),
	\label{EQ03}
\end{align}
where $\theta\colon[0,\infty)\to [0,1]$ is a smooth function such that $\theta\equiv 1$ on $[0,2]$ and $\theta\equiv 0$ on~$[3,\infty)$, and $\bar{\epsilon}$ be a small positive constant to be determined. We will solve the truncated model using nested fixed point arguments, and as a first step, we make the following claim.
\cole
\begin{Lemma}\label{L03}
	Let $T>0$, $n \in \mathbb{N}$, $v \in L^2(\Omega; C([0,T],H^\frac{1}{2})) \cap L^2(\Omega; L^2([0,T], H^\frac{3}{2}))$, and $u_0 \in L^2(\Omega; H^\frac{1}{2})$. Suppose that both $v$ and $u_0$ are divergence- and average-free, and that $\epsilon_b>0$ in \eqref{EQ14} is sufficiently small.
	Then, the initial value problem
		\begin{align}
		\begin{split}
			&	(\partial_t - \Delta) u 
			= -\psi^2(v)\mathcal{P} \nabla \cdot  (v \otimes v) 
			+ \mathcal{P}( (b\cdot \nabla)  u )\dot{W},
			\\
			&\nabla\cdot u = 0,\\
			&u (0)=u_0\comma t\geq 0, ~x\in \TT^3
		\end{split}\label{EQ05}
	\end{align}
	has a unique divergence- and average-free (probabilistically) strong solution 
	$u \in L^2(\Omega; C([0,T],H^\frac{1}{2})) \cap L^2(\Omega; L^2([0,T], H^\frac{3}{2}))$. 
\end{Lemma}
\colb





\begin{proof}[Proof of Lemma~\ref{L03}]
	Assuming that $u^{(-1)}=0$ and letting $m \in \mathbb{N}_0$, we consider
\begin{align}
	\begin{split}
	&(\partial_t - \Delta) \um 
	= -\psi^2(v) \mathcal{P} \nabla \cdot  (v \otimes v) 
	+ \mathcal{P}( (b\cdot \nabla)  u^{(m-1)}) \dot{W},
		\\
	&\nabla\cdot \um = 0,\\
	&\um (0)=u_0\comma t\geq 0, ~x\in \TT^3
	.
	\end{split}\label{EQ12}
\end{align}
Recall the Sobolev product inequality~\eqref{EQ15} and the definition of the cutoff function~\eqref{EQ03}. It follows by \cite[Lemma~3.1]{AKX} that \eqref{EQ12} admits a unique solution in $L^2(\Omega; C([0,T],H^\frac{1}{2})) \cap L^2(\Omega; L^2([0,T], H^\frac{3}{2}))$ when $m=0$, and this solution is divergence- and average-free. A standard induction argument extends this result to all $m \in \NNp$, yielding unique solutions in the same space, which remain divergence- and average-free.

Moreover, we obtain uniform-in-$m$ bounds. Specifically, by~\cite[Lemma~3.1]{AKX},
\begin{align}
	\begin{split}
		\EE\biggl[&
		\sup_{0\leq t\leq T}
		\Vert \um(t)\Vert_{H^{\frac12 }}^2
		+ \int_0^{T} \Vert \um(t)\Vert_{H^{\frac{3}{2}}}^2 \,dt\biggr]
		\\&\indeq
		\leq 
		C\EE\biggl[  \Vert u_0\Vert_{H^{\frac12}}^2\biggr]
		+C\EE\biggl[
		\int_0^{T} \psi^4(v) \Vert v \otimes v\Vert_{H^{\frac12}}^2
		\,dt
		+
		\int_0^{T}\Vert (b\cdot \nabla)  u^{(m-1)}\Vert_{H^{\frac12 }}^2
		\,dt
		\biggr]
		,
	\end{split}
	\label{EQ13}
\end{align}
where the constant $C$ is independent of time $T$ and the function $v$. For the quadratic term, we utilize~\eqref{EQ15} and~\eqref{EQ03}, obtaining
\begin{align}
	\EE\biggl[
	\int_0^{T} \psi^4(v) \Vert v \otimes v\Vert_{H^{\frac12}}^2
	\,dt
	\biggr]
	\lec
	 \EE\biggl[
	 \int_0^{T} \psi^4(v) \Vert v\Vert_{H^{\frac{1}{2}}}^2 \Vert v\Vert_{H^{\frac{3}{2}}}^2
	 \,dt
	 \biggr]
	 \lec
	  \bar{\epsilon}^2 \EE\biggl[
	  \int_0^{T} \psi^4(v) \Vert v\Vert_{H^{\frac{3}{2}}}^2
	  \,dt
	  \biggr]
	   \lec
	    \bar{\epsilon}^4
	    ,\label{EQ16}
\end{align}
while for the transport noise coefficient, we exploit~\eqref{EQ14}, deriving
\begin{align}
	\EE \biggl[\int_0^{T}\Vert (b \cdot \nabla) u^{(m-1)}\Vert_{H^{\frac12 }}^2
	\,dt
	\biggr]
	\lec
	 \epsilon_b^2 
	  \EE \biggl[\int_0^{T}\Vert u^{(m-1)}\Vert_{H^{\frac{3}{2} }}^2
	  \,dt
	  \biggr]
	.\label{EQ17}
\end{align}
Combining \eqref{EQ13}--\eqref{EQ17}, we arrive at
\begin{align}
	\begin{split}
		\EE\biggl[&
		\sup_{0\leq t\leq T}
		\Vert \um(t)\Vert_{H^{\frac12 }}^2
		+ \int_0^{T} \Vert \um(t)\Vert_{H^{\frac{3}{2}}}^2 \,dt\biggr]
		\\&\indeq
		\le 
		C\EE \biggl[\Vert u_0\Vert_{H^{\frac12}}^2\biggr] 
		+	C\bar{\epsilon}^4+C\epsilon_b^2 
		\EE \biggl[\int_0^{T}\Vert u^{(m-1)}\Vert_{{H}^{\frac{3}{2} }}^2
		\,dt
		\biggr],		
	\end{split}
	\llabel{EQ18}
\end{align}
for some $C$ that is independent of $T$, $v$, and $m$. We rewrite this as
\begin{align}
	K^{(m)} \le C\EE \biggl[\Vert u_0\Vert_{H^{\frac12}}^2\biggr]  + C\bar{\epsilon}^4
	 +C \epsilon_b^2 K^{(m-1)}
	 ,\llabel{EQ19}
\end{align}
where 
\begin{align}
	K^{(m)} = \EE\biggl[&
	\sup_{0\leq t\leq T}
	\Vert \um(t)\Vert_{H^{\frac12 }}^2
	+ \int_0^{T} \Vert \um(t)\Vert_{H^{\frac{3}{2}}}^2 \,dt\biggr]
	.\llabel{EQ20}
\end{align}
Therefore, when $\epsilon_b$ is sufficiently small,
we obtain that
\begin{align}
	K^{(m)} \le \frac{C\EE \biggl[\Vert u_0\Vert_{H^{\frac12}}^2\biggr]  + C\bar{\epsilon}^4}{1-C\epsilon_b^2}\lec \EE \biggl[\Vert u_0\Vert_{H^{\frac12}}^2\biggr]  + \bar{\epsilon}^4	\comma m \in \mathbb{N}. \label{EQ21}
\end{align}

Now, we establish the contraction property of $\{\um\}$. Taking the difference of \eqref{EQ12}, we write
\begin{align}
	(\partial_t - \Delta) U^{(m)} 
	= \mathcal{P} ((b\cdot \nabla) U^{(m-1)}) \dot{W}
	,\label{EQ22}
\end{align}
where $U^{(m)} = \um - u^{(m-1)}$, and $U^{(m)}(0) = 0$.
We employ~\cite[Lemma~3.1]{AKX} and obtain
\begin{align}
	\begin{split}
		\EE\biggl[ &
		\sup_{0\leq t\leq T}
		\Vert U^{(m)}(t)\Vert_{H^{\frac12 }}^2
		+ \int_0^{T} \Vert U^{(m)}(t)\Vert_{H^{\frac{3}{2}}}^2 \,dt\biggr]
		\\&\indeq
		\leq 
		C\EE \biggl[
		\int_0^{T}\Vert (b \cdot \nabla) U^{(m-1)}\Vert_{H^{\frac12 }}^2
		\,dt
		\biggr]
		\le C \epsilon_b^2 
		\EE \biggl[
		\int_0^{T}\Vert U^{(m-1)}\Vert_{H^{\frac{3}{2}}}^2
		\,dt
		\biggr]
		,\label{EQ23}
	\end{split}
\end{align}
from where choosing $\epsilon_b$ sufficiently small yields a fixed point $u$ of~\eqref{EQ12}, with the convergence being exponentially rapid. This fixed point obeys the bound in~\eqref{EQ21} with respect to all $T>0$ and all $v$. To claim that it is a divergence- and average-free solution, we may pass to the limit in the analytic weak formulations of~\eqref{EQ12}$_1$ and \eqref{EQ12}$_2$ using the exponential rate of convergence. To establish the pathwise uniqueness of the solutions, we may proceed as in \eqref{EQ22} and \eqref{EQ23}. This concludes the proof.
\end{proof}

With Lemma~\ref{L03}, we now establish the solvability of the truncated model of~\eqref{EQ01}, as detailed in the following lemma.

\cole
\begin{Lemma}\label{L04}
	Let $T>0$ and $u_0 \in L^2(\Omega; H^\frac{1}{2})$ be divergence- and average-free. 
	Moreover, assume that $\bar{\epsilon}>0$ in \eqref{EQ03} and $\epsilon_b>0$ in \eqref{EQ14} are sufficiently small.
Then, the initial value problem
\begin{align}
	\begin{split}
		&	(\partial_t - \Delta) u 
		= -\psi^2(u)\mathcal{P} \nabla \cdot  (u \otimes u) 
		+ \mathcal{P}( (b\cdot \nabla) u )\dot{W}
		\\
		&\nabla\cdot u = 0,\\
		&u (0)=u_0\comma t\geq 0, ~x\in \TT^3
	\end{split}\label{EQ02}
\end{align} 
has a unique divergence- and average-free strong solution
	$u \in L^2(\Omega; C([0,T],H^\frac{1}{2})) \cap L^2(\Omega; L^2([0,T], H^\frac{3}{2}))$.
\end{Lemma}
\colb

\begin{proof}[Proof of Lemma~\ref{L04}]
	We consider the iterative scheme for \eqref{EQ02}:
	\begin{align}
		(\partial_t - \Delta) \um 
		= -\psi^2(u^{(m-1)}) \nabla \cdot \mathcal{P} (u^{(m-1)} \otimes u^{(m-1)}) 
		+ \mathcal{P} ((b \cdot \nabla)  \um) \dot{W}
		,\label{EQ24}
	\end{align}
where $m \in \mathbb{N}_0$, and $u^{(-1)}:=0$.
	By an induction argument and Lemma~\ref{L03}, \eqref{EQ24} has a unique incompressible solution $\um \in  L^2(\Omega; C([0,T],H^\frac{1}{2})) \cap L^2(\Omega; L^2([0,T], H^\frac{3}{2})$ with initial data $u_0$
	for all $m \in \mathbb{N}_0$. Moreover, $\um$ are average-free and uniformly bounded in $m$.

	To identify the fixed point of the sequence $\{ \um\}$ in $L^2(\Omega; C([0,T],H^\frac{1}{2})) \cap L^2(\Omega; L^2([0,T], H^\frac{3}{2})$, we write $U^{(m)} = \um - u^{(m-1)}$ and consider the difference of the equation~\eqref{EQ24}, 
	\begin{align}
		\begin{split}
		(\partial_t - \Delta) U^{(m)}
		&=- \bigl(\psi(\umm)-\psi(u^{(m-2)})\bigr) \psi(\umm ) \nabla \cdot \mathcal{P} (\umm  \otimes \umm )
		\\&\indeq - \psi(u^{(m-2)}) \psi(\umm ) \nabla \cdot \mathcal{P} (U^{(m-1)} \otimes \umm )
		\\&\indeq  -\psi(u^{(m-2)}) \psi(\umm ) \nabla \cdot \mathcal{P} (u^{(m-2)} \otimes U^{(m-1)})
		 \\&\indeq  -\psi(u^{(m-2)}) \bigl(\psi(\umm )-\psi(u^{(m-2)})\bigr) \nabla \cdot \mathcal{P} (u^{(m-2)} \otimes u^{(m-2)})
		\\&\indeq+ \mathcal{P} ((b \cdot \nabla) U^{(m)})) \dot{W}  
		=I_1+\ldots +I_5.\label{EQ27}
		\end{split}
	\end{align}
	We start by estimating $I_1$. Utilizing~\eqref{EQ15}, we derive
	\begin{align}
		\begin{split}
		\EE \biggl[ \int_0^T \Vert I_1\Vert_{H^{-\frac{1}{2}}}^2\,dt \biggr]
		 &\lec
		  \EE \biggl[ \sup_{0\leq t\leq T} \bigl| \psi(\umm )-\psi(u^{(m-2)} ) \bigr|^2
		   \int_0^T \psi^2(\umm ) \Vert \umm \Vert_{H^{\frac12}}^2\Vert u^{(m-2)}\Vert_{{H^\frac{3}{2}}}^2\,dt \biggr]
		  \\ &\lec
		   \bar{\epsilon}^2 
		    \EE \biggl[ \sup_{0\leq t\leq T} \bigl| \psi(\umm )-\psi(u^{(m-2)}) \bigr|^2
		   \int_0^T \psi^2(\umm ) \Vert \umm \Vert_{{H^\frac{3}{2}}}^2\,dt \biggr]
		     \\ &\lec 
		      \bar{\epsilon}^4 
		      \EE \biggl[ \sup_{0\leq t\leq T} \bigl| \psi(\umm )-\psi(u^{(m-2)}) \bigr|^2 \biggr]
		      \\ &\lec 
		      \bar{\epsilon}^2 
		      \EE \biggl[ \sup_{0\leq t\leq T} \bigl\Vert U^{(m-1)} \bigr\Vert_{H^{\frac12}}^2
		       +\int_0^T \Vert U^{(m-1)}\Vert_{H^{\frac{3}{2}}}^2 \,dt\biggr]
		   .\label{EQ28}
		   \end{split}
	\end{align}
	A similar computation for $I_4$ yields
	\begin{align}
		\EE \biggl[ \int_0^T \Vert I_4\Vert_{H^{-\frac{1}{2}}}^2\,dt\biggr]
		\lec   \bar{\epsilon}^2 
		\EE \biggl[  \sup_{0\leq t\leq T} \bigl\Vert U^{(m-1)} \bigr\Vert_{H^{\frac12}}^2
		+\int_0^T \Vert U^{(m-1)}\Vert_{H^{\frac{3}{2}}}^2\,dt \biggr]
		.\label{EQ29}
	\end{align}

	For $I_2$, we have instead
	\begin{align}
		\begin{split}
			\EE \biggl[ \int_0^T \Vert I_2\Vert_{H^{-\frac{1}{2}}}^2\biggr]
			&\lec
			\EE \biggl[ 
			\int_0^T \psi^2(u^{(m-1)})\psi^2(u^{(m-2)}) 
			\left(\Vert U^{(m-1)}\Vert_{H^{\frac12}}^2\Vert u^{(m-1)}\Vert_{{H^\frac{3}{2}}}^2
			+\Vert U^{(m-1)}\Vert_{H^{\frac{3}{2}}}^2\Vert u^{(m-1)}\Vert_{{H^\frac{1}{2}}}^2
			\right) \biggr]
			\\ &\lec 
			\EE \biggl[ \sup_{0\leq t\leq T} \Vert U^{(m-1)}\Vert_{H^{\frac12}}^2 
			\int_0^T \psi^2(u^{(m-1)}) 
			\Vert u^{(m-1)}\Vert_{{H^\frac{3}{2}}}^2
			\biggr]
			+\bar{\epsilon}^2
			\EE \biggl[ 
			\int_0^T  
			\Vert U^{(m-1)}\Vert_{H^{\frac{3}{2}}}^2
			\biggr]
			\\ & \lec 
			\bar{\epsilon}^2 
			\EE \biggl[ \sup_{0\leq t\leq T} \Vert U^{(m-1)}\Vert_{H^{\frac12}}^2
			+\int_0^T \Vert U^{(m-1)}\Vert_{H^{\frac{3}{2}}}^2 \biggr]
			.\label{EQ30}
		\end{split}
	\end{align}
	\colb
	The same bound holds for $I_3$, i.e.,
	\begin{align}
		\EE \biggl[ \int_0^T \Vert I_3\Vert_{H^{-\frac{1}{2}}}^2\,dt\biggr]
		\lec 
		\bar{\epsilon}^2 
		\EE \biggl[ \sup_{0\leq t\leq T} \Vert U^{(m-1)}\Vert_{H^{\frac12}}^2
		+\int_0^T \Vert U^{(m-1)}\Vert_{H^{\frac{3}{2}}}^2 \biggr]
		.\label{EQ31}
	\end{align}
	\colb
	
	Finally, for the term~$I_5$ we have
	\begin{align}
		\EE \biggl[ \int_0^T \Vert \mathcal{P} ((b \cdot \nabla) U^{(m)}))\Vert_{H^\frac{1}{2}}^2\,dt\biggr]
		\lec 
		\epsilon_b^2 
		\EE \biggl[ \int_0^T \Vert U^{(m)}\Vert_{H^{\frac{3}{2}}}^2\,dt \biggr]
		.\label{EQ32}
	\end{align}
	It then follows from~\eqref{EQ28}--\eqref{EQ32} and ~\cite[Lemma~3.1]{AKX} that
	 \begin{align*}
	 	\begin{split}
	 	&\EE \biggl[  \sup_{0\leq t\leq T} \bigl\Vert U^{(m)} \bigr\Vert_{H^{\frac12}}^2
	 	+\int_0^T \Vert U^{(m)}\Vert_{H^{\frac{3}{2}}}^2\,dt \biggr]
	 	\\ & \indeq \indeq 
	 	\lec \bar{\epsilon}^2 
	 	\EE \biggl[  \sup_{0\leq t\leq T} \bigl\Vert U^{(m-1)} \bigr\Vert_{H^{\frac12}}^2
	 	+\int_0^T \Vert U^{(m-1)}\Vert_{H^{\frac{3}{2}}}^2\,dt \biggr]+\epsilon_b^2 
	 	\EE \biggl[  \int_0^T \Vert U^{(m)}\Vert_{H^{\frac{3}{2}}}^2\,dt \biggr],
	 	\end{split}
	 \end{align*}
	and the desired fixed point $u$ is obtained by choosing $\bar{\epsilon}$ and $\epsilon_b$ sufficiently small. Moreover, the convergence to this fixed point is exponentially fast. Now, utilizing the exponential rate of convergence, we may pass to the limit in 
	\begin{align}
		\begin{split}
			&(u^{(m)}( t),\varphi)
			= ( u_0,\varphi)+\int_0^t \bigl((u^{(m)}, \Delta \varphi) 
			+ ( \psi^2(\umm ) \mathcal{P}(\umm  \otimes \umm ), \nabla \varphi)\bigr)\,ds
			\\& \indeq \indeq \indeq \indeq
			-\int_0^t (\mathcal{P}(b \otimes \um), \nabla \varphi)\,d\WW_s
		,
			\\
			&(u^{(m)}( t), \nabla \varphi)
			= 0	\comma (t,\omega)\text{-a.e.},
		\end{split}
		\llabel{EQ33}
	\end{align}
	where $\varphi\in C^{\infty}(\TT^3)$, and $t\in [0, T]$.

	The convergence of the non-noise terms is ensured by the dominated convergence theorem and the Sobolev product inequality~\eqref{EQ15}, after the following splitting,
	\begin{align}
		\begin{split}
		 &\left| 
			\int_0^T \left(\psi^2(\umm ) \mathcal{P} (\umm \otimes \umm )
			-\psi^2(u) \mathcal{P} (u\otimes u), \nabla \varphi\right) \,dt
			\right|
			\\ &\lec
			\left| 
			\int_0^T \left( (\psi(\umm ) -\psi(u))
			\psi(\umm ) \mathcal{P} (\umm \otimes \umm )
			, \nabla \varphi\right) \,dt\right|
			\\ &\indeq+
		\left| 
			\int_0^T \left(  \psi(u) \psi(\umm ) 
			\mathcal{P} \bigl((\umm  -u) \otimes \umm \bigr), \nabla \varphi\right) \,dt\right|
			\\ &\indeq +
			\left| 
			\int_0^T \left(  \psi(u) \psi(\umm ) 
			\mathcal{P} \bigl(u \otimes (\umm -u)\bigr), \nabla \varphi\right) \,dt\right|
			\\ &\indeq +
			\left| 
			\int_0^T \left(  \psi(u) \bigl(\psi(\umm )-\psi(u)\bigr) 
			\mathcal{P} \left(u \otimes u\right), \nabla \varphi\right) \,dt\right|.
\llabel{EQ35}
		\end{split}
	\end{align}

	For the noise term, we apply the Burkholder-Davis-Gundy (BDG) inequality, concluding
	 \begin{align}
		\begin{split}
			\mathbb{E}&\biggl[\sup_{0\le t \le T}
			\biggl|
			\int_0^t \left(\mathcal{P}(b \otimes (\um-u)), \nabla \varphi\right)\,d\WW_s
			\biggr|  \biggr]
			\lec
			\mathbb{E}\biggl[
			\biggl(
			\int_0^{T} \Vert \um-u\Vert_{L^{2}}^2 \,dt
			\biggr)^{1/2}\biggr]
			\to 0, 
			\llabel{EQ34}
		\end{split}
	\end{align}
	as $m \to \infty$.
In addition, since $\um$ is average-free, so is $u$.

	We may follow the proof of the contraction property to obtain the pathwise uniqueness of the solution.
\end{proof}

Now, we present a refined energy estimates for the solutions of \eqref{EQ02}. Define
\begin{align}
	{Q}(t)^2
	:=\Vert u(t)\Vert_{H^{\frac{1}{2}}}^2 
	+ \int_0^t \Vert u(s)\Vert_{H^{\frac{3}{2}}}^2 \,ds
	.\label{EQ09}
\end{align}

\cole
\begin{Lemma}\label{L08}
	Let $T>0$ and $u_0 \in L^2(\Omega; H^\frac{1}{2})$ be divergence- and average-free. 
	Moreover, assume that $\bar{\epsilon},~\epsilon_b>0$ in \eqref{EQ03} and \eqref{EQ14} are sufficiently small.
	Then, the pathwise unique solution 
	$u \in L^\infty(\Omega; C([0,T],H^\frac{1}{2})) \cap L^\infty(\Omega; L^2([0,T], H^\frac{3}{2}))$
	of \eqref{EQ02} satisfies 
	\begin{align}
		\EE \biggl[\sup_{0 \le t \le T} Q(t)^2 - \Vert u_0\Vert_{H^{\frac12}}^2 \biggr] 
		\lec 
		\epsilon_b^2 \EE \biggl[ \int_0^T \Vert u(t) \Vert_{H^{1}}^2\,dt\biggr]
		\label{EQ100}
	\end{align} 
	and
	\begin{align}
		\EE \biggl[\sup_{0 \le t \le T} Q(t)^2 \biggr] 
		\lec 
		\EE \biggl[ \Vert u_0\Vert_{H^{\frac12}}^2 \biggr] 
		,\label{EQ107}
	\end{align} 
	where the implicit constant is independent of $T$.
\end{Lemma}
\colb

\begin{proof}[Proof of Lemma~\ref{L08}]
	Denote by $\bar{\Lambda}$ the non-homogenous differential operator whose Fourier symbol is $(1+|k|^2)^{1/2}$, and let $T>0$ be fixed.
	We employ It\^o's formula, obtaining
	\begin{align}
		\begin{split}
			&\Vert\bar{\Lambda}^{\frac12}u(t)\Vert_{L^2}^{2}
			- \Vert\bar{\Lambda}^{\frac12}(u_0)\Vert_{L^2}^{2} 
			+ 2\int_0^t \int_{\TT^3} | \bar{\Lambda}^{\frac{3}{2}} u(s)|^2 \,dx ds
			\\&\indeq
			\leq
			2\int_0^t\left|  (\psi^2(u)\mathcal{P} \nabla \cdot  (u \otimes u),\bar{\Lambda} u)\right| \,ds
			\\&\indeq\indeq
			+ \int_0^t\int_{\TT^3} \vert \bar{\Lambda}^{\frac12} \mathcal{P}( (b\cdot \nabla) u )\vert^2\,dx ds
			+ 2 
			\left|\int_0^t( \mathcal{P}( (b \cdot \nabla) u ),\bar{\Lambda} u(s)) \,d\WW_s\right|
			\\&\indeq
			= 
			I_1 + I_2 + I_3
			\commaone t\in [0,T].    
		\end{split}
		\label{EQ102}
	\end{align}
	Upon using Young's inequality and the properties of cutoff, we estimate $I_1$ as
	\begin{align}
		I_1 \le C_\eta \bar{\epsilon}^2 \int_0^t \Vert u\Vert_{H^{\frac{3}{2}}}^2
		+ \eta \int_0^t \Vert u\Vert_{H^{\frac{3}{2}}}^2
		,\llabel{EQ103}
	\end{align}
	where $\eta>0$, while for $I_2$, we have
	\begin{align}
		I_2 \lec \epsilon_b^2 \int_0^t \Vert u\Vert_{H^{\frac{3}{2}}}^2
		.\llabel{EQ104}
	\end{align}
We choose $\eta$, $\bar{\epsilon}$, and $\epsilon_b$ sufficiently small so that the left-hand side of~\eqref{EQ102} absorbs $I_1 + I_2$, obtaining
	\begin{align}
		\Vert u(t)\Vert_{H^\frac{1}{2}}^{2}
		- \Vert u_0\Vert_{H^\frac{1}{2}}^{2} 
		+ \int_0^t \Vert u(t)\Vert_{H^\frac{1}{2}}^{2} \,ds
		\lec I_3
		\commaone t\in [0,T].    
		\llabel{EQ105}
	\end{align}
	Then, we take the supremum in $t$ and expectation in $\omega$, and utilize the BDG inequality for $I_3$,
	obtaining
	\begin{align}
		\EE \biggl[\sup_{0 \le t \le T} Q(t)^2 - \Vert u_0\Vert_{H^\frac{1}{2}}^{2}\biggr]
		\lec
		\EE \biggl[ \sup_{0 \le t \le T} I_3 \biggr] \lec
		\epsilon_b^2 \EE \biggl[ \int_0^T \Vert u (t)\Vert_{H^{1}}^2\,dt \biggr]
		,\llabel{EQ106}
	\end{align}
	where the implicit constant is independent of $T$. Finally, \eqref{EQ107} follows by choosing $\epsilon_b$ sufficiently small and absorbing the right-hand side of \eqref{EQ100}, using the embedding $H^{\frac{3}{2}} \hookrightarrow H^1$.
\end{proof}

To remove the impact of the cutoff function, we introduce the stopping time:
\begin{align}
	\tau = \inf \{ t>0 : Q(t) > \bar{\epsilon}\}
	,\label{EQ10}
\end{align}
where $\bar{\epsilon}$ is the parameter appearing in the definition of the cutoff functions in \eqref{EQ03}. It then follows from this definition that
\begin{align}
	\sup_{0 \le t \le \tau} Q(t)^2 \le \bar{\epsilon}^2
	\comma \PP \text{ a.s.} \label{EQ60}
\end{align} 

Now, we claim that the stopping time $\tau$ is almost surely positive.

\cole
\begin{Lemma}\label{L06}
	Assume that the parameters $\epsilon_b , ~\bar{\epsilon} > 0$ in \eqref{EQ14} and \eqref{EQ03} are sufficiently small, and that the initial data satisfy $\sup_{\Omega}\|u_0\|_{H^{1/2}} \le \epsilon_0$ with $\epsilon_0$ small relative to $\bar{\epsilon}$. Then the stopping time $\tau$ defined in \eqref{EQ10} is positive $\PP$-almost surely.
\end{Lemma}
\colb

\begin{proof}[Proof of Lemma~\ref{L06}]
	Let $\delta>0$. We employ \eqref{EQ100} with $T$ replaced by $\delta$, obtaining that
	\begin{align}
		\begin{split}
			\EE\biggl[&
			\sup_{0\leq t\leq \delta}
			\left(\Vert u(t)\Vert_{H^{\frac12 }}^2
			- \Vert u_0\Vert_{H^{\frac12}}^2
			+ \int_0^{t} \Vert u(t)\Vert_{H^{\frac{3}{2}}}^2 \,dt\right)\biggr]
			\lec
			\epsilon_b^2 \EE\biggl[
			\int_0^{\delta} \Vert u(t)\Vert_{H^{1}}^2\,dt
			\biggr]
			,
		\end{split}
		\label{EQ48}
	\end{align}
	from where using the Sobolev interpolation inequality, 
	\begin{align}
		\Vert f\Vert_{{H}^{1}}^2
		\leq 
		\Vert f\Vert_{H^{\frac12}}
		\Vert f\Vert_{H^{\frac{3}{2}}}
		,\label{EQ51}
	\end{align}
	we arrive at
	\begin{align}
		\begin{split}
			\EE\biggl[&
			\sup_{0\leq t\leq \delta}\left(
			\Vert u(t)\Vert_{H^{\frac12 }}^2
			- \Vert u_0\Vert_{H^{\frac12}}^2
			+ \int_0^{t} \Vert u(t)\Vert_{H^{\frac{3}{2}}}^2 \,dt\right)\biggr]
			\lec
			\epsilon_b^2 \EE\biggl[ 
			\int_0^{\delta} \Vert u(t)\Vert_{H^{\frac{1}{2}}}\Vert u(t)\Vert_{H^{\frac{3}{2}}}
			\,dt
			\biggr]
				.
		\end{split}	\label{EQ49}
	\end{align}
	Next, we employ H\" older's inequality and derive 
	\begin{align}
		\begin{split}
			\EE\biggl[&
		\sup_{0\leq t\leq \delta}\left(
		\Vert u(t)\Vert_{H^{\frac12 }}^2
		- \Vert u_0\Vert_{H^{\frac12}}^2
		+ \int_0^{t} \Vert u(t)\Vert_{H^{\frac{3}{2}}}^2 \,dt\right)\biggr]
		\lec
		\epsilon_b^2 \EE\biggl[ \sup_{0 \le t \le \delta} \Vert u\Vert_{H^{\frac12}}
		\int_0^{\delta} \Vert u(t)\Vert_{H^{\frac{3}{2}}}
		\,dt
		\biggr]
		\\&\indeq\lec
		\epsilon_b^2 \left( \EE \biggl[ \sup_{0 \le t \le \delta} \Vert u(t)\Vert_{H^{\frac{1}{2}}}\biggr]^2\right)^\frac12
		\left( \EE \biggl[ \int_0^\delta \Vert u(t)\Vert_{H^{\frac{3}{2}}}\,dt\biggr]^2\right)^\frac12
		\\&\indeq\lec \epsilon_b^2 \sqrt{\delta}
		\left( \EE \biggl[ \sup_{0 \le t \le \delta} \Vert u(t)\Vert_{H^{\frac{1}{2}}}^2 \biggr]\right)^\frac12
		\left( \EE \biggl[ \int_0^\delta \Vert u(t)\Vert_{H^{\frac{3}{2}}}^2 \,dt\biggr]\right)^\frac12
		.\end{split}
		\label{EQ53}
	\end{align}
	Finally, utilizing \eqref{EQ107} together with the assumption~$\sup_{\Omega}\|u_0\|_{H^{1/2}} \le \epsilon_0$, we conclude
	\begin{align}
		\begin{split}
			\EE\biggl[&
		\sup_{0\leq t\leq \delta}\left(
		\Vert u(t)\Vert_{H^{\frac12 }}^2
		- \Vert u_0\Vert_{H^{\frac12}}^2
		+ \int_0^{t} \Vert u(t)\Vert_{H^{\frac{3}{2}}}^2 \,dt\right)\biggr]
		\lec
		\epsilon_b^2 \epsilon_0^2 \sqrt{\delta}
		.
		\end{split}\label{EQ56}
	\end{align}

Now, we choose $\epsilon_0^2<\bar{\epsilon}^2/2$ and apply Markov's inequality obtaining,
	\begin{align}
		\begin{split}
			\PP(\tau &< \delta)
			\le \PP \left(\sup_{0\le t \le \delta} Q (t)^2\ge \bar{\epsilon}^2\right)
			\le \PP \left(\sup_{0\le t \le \delta} Q (t)^2 -  \Vert u_0\Vert_{H^{\frac12}}^2\ge \bar{\epsilon}^2 - \epsilon_0^2\right)
			\\ & \quad
			 \le  \frac{\EE\biggl[ \sup_{0\le t \le \delta} Q (t)^2-\Vert u_0\Vert_{H^{\frac12}}^2\biggr]}
			{\bar{\epsilon}^2- \epsilon_0^2}
			\lec \sqrt{\delta}  
			.\llabel{EQ47}
		\end{split}
	\end{align} 
	Then, $\PP(\tau =0)=\lim_{n\to \infty} \PP(\tau <\frac1n)=0$,
	which concludes the proof. 
\end{proof}

Clearly, the non-truncated model~\eqref{EQ01}$_1$ coincides with the truncated model~\eqref{EQ02}$_1$ up to time $\tau$. Therefore, assuming that $\sup_{\omega}\Vert u_0\Vert_{H^{\frac12}}\le \epsilon_0$, with $\epsilon_0$ small relative to $\bar{\epsilon}$, we may apply Lemma~\ref{L04} and conclude that \eqref{EQ01} admits a probabilistically strong solution that is divergence-free and mean-free at least up to time $\tau$. Moreover, the energy estimates \eqref{EQ65}–\eqref{EQ66} for this solution follow from \eqref{EQ100}, \eqref{EQ107}, and \eqref{EQ60}. It remains to establish pathwise uniqueness under the smallness assumption.
\begin{Lemma}\label{L05}
	Assume that the parameters $\bar{\epsilon}, ~\epsilon_b > 0$ in \eqref{EQ03} and \eqref{EQ14} are sufficiently small. If $(\tilde{u}, \tilde{\tau})$ and $(u, \tau)$ are both solutions to \eqref{EQ01} satisfying \eqref{EQ65} and \eqref{EQ66}, then 
	\[
	\PP\bigl(u(t)=\tilde{u}(t) \text{ for all } t\in [0, \tau\wedge \tilde{\tau} ] \bigr)=1.
	\]
\end{Lemma}
\begin{proof}[Proof of Lemma~\ref{L05}]
We apply ~\cite[Lemma~3.1]{AKX} to the difference model and derive the estimates 
    \begin{align}
    	\begin{split}
    		\EE\biggl[&
    		\sup_{0\leq t\leq \tau \wedge \tilde{\tau}}
    		\Vert u(t)-\tilde{u}(t)\Vert_{H^{\frac12 }}^2
    		+ \int_0^{\tau \wedge \tilde{\tau}} \Vert u(t)-\tilde{u}(t)\Vert_{H^{\frac{3}{2}}}^2 \,dt\biggr]
    			\\&\indeq
    		\lec \EE\biggl[
    		\int_0^{\tau \wedge \tilde{\tau}}  \Vert u \otimes (u-\tilde{u})\Vert_{H^{\frac12}}^2
   +
   \Vert \tilde{u} \otimes (u-\tilde{u})\Vert_{H^{\frac12}}^2
    		+
    	\Vert (b\cdot \nabla)  (u-\tilde{u})\Vert_{H^{\frac12 }}^2
    		\,dt
    		\biggr]
    		\\&\indeq
    		\lec (\bar{\epsilon}^2+\epsilon_b^2)
    		\EE\biggl[
    		\sup_{0\leq t\leq \tau \wedge \tilde{\tau}}
    		\Vert u(t)-\tilde{u}(t)\Vert_{H^{\frac12 }}^2
    		+ \int_0^{\tau \wedge \tilde{\tau}} \Vert u(t)-\tilde{u}(t)\Vert_{H^{\frac{3}{2}}}^2 \,dt\biggr]
    		,
    	\end{split}
    	\llabel{EQ61}
    \end{align}
    yielding pathwise uniqueness of the solution. 
   \end{proof} 

We note that neither the bounds established for the solutions nor the smallness assumptions on $\bar{\epsilon}$ and $\epsilon_b$ depend on $T$. Therefore, it is reasonable to expect that the probability $\tau = \infty$ is large.

\cole
\begin{Lemma}\label{L07}
	Assume that the parameters $\bar{\epsilon}, ~\epsilon_b $ in \eqref{EQ03} and \eqref{EQ14} are sufficiently small. Then, for any $p_0 \in (0,1)$, there exists $\epsilon_0 > 0$, chosen sufficiently small relative to $\bar{\epsilon}$, such that if $\,\sup_{\Omega}\Vert u_0\Vert_{H^{1/2}} \le \epsilon_0$, then $\PP(\tau < \infty) \le p_0$.
\end{Lemma}
\colb

\begin{proof}[Proof of Lemma~\ref{L07}]
	Let $T>0$. Recall~\eqref{EQ100}--\eqref{EQ107}.
	Then, the embedding $H^{\frac{3}{2}}\hookrightarrow  H^1$ and 
	the assumption on $u_0$ imply
	\begin{align}
		\begin{split}
			\EE\biggl[&
			\sup_{0\leq t\leq T}
			\left(\Vert u(t)\Vert_{H^{\frac12 }}^2
			-\Vert u_0 \Vert_{H^{\frac12 }}^2
			+ \int_0^{t} \Vert u(s)\Vert_{H^{\frac{3}{2}}}^2 \,ds\right)\biggr]
			\lec \epsilon_b^2 \epsilon_0^2
			,
		\end{split}
		\llabel{EQ56}
	\end{align}
	where the implicit constant is independent of $T$. Assuming that $\epsilon_0$ is relatively small compared to $\bar{\epsilon}$,
	we have
	\begin{align}
		\begin{split}
			\PP(\tau &< T)
			\le \PP \left(\sup_{0\le t \le T} Q (t)^2\ge \bar{\epsilon}^2\right)
			\le \PP \left(\sup_{0\le t \le T} Q(t)^2 - \Vert u_0\Vert_{H^{\frac12}}^2\ge \bar{\epsilon}^2 - \epsilon_0^2\right)
			\\& \le  \frac{ \EE\biggl[ \sup_{0\le t \le T} Q(t)^2 -\Vert u_0\Vert_{H^{\frac12}}^2\biggr]}{\bar{\epsilon}^2- \epsilon_0^2}
			\lec \frac{\epsilon_b^2 \epsilon_0^2}{\bar{\epsilon}^2-\epsilon_0^2}  
			\lec \frac{\epsilon_0^2}{\bar{\epsilon}^2-\epsilon_0^2}
			\leq p_0\llabel{EQ58}
		\end{split}
	\end{align}
	for any predesignated $p_0\in (0,1)$, upon choosing $\epsilon_0$ sufficiently small. Since $T$ is arbitrary, we conclude $\PP(\tau < \infty) \le p_0$.
\end{proof}

\section*{Acknowledgments}
\rm
MSA was supported in part by the NSF grant DMS-2205493 and the USC Dornsife Summer Research Award.

\ifnum\sketches=1

\fi

\end{document}